\documentclass[12pt,a4paper,reqno]{amsart}
\usepackage{amsmath,amsfonts,amssymb,amsthm,color}  
\usepackage{mathrsfs} 
\usepackage{MnSymbol}

\usepackage[utf8]{inputenc}
\usepackage{graphicx}
\usepackage{bibentry}
\usepackage[colorlinks,citecolor=red,pagebackref,hypertexnames=false]{hyperref}

\theoremstyle{plain}
\newtheorem{theorem}{Theorem}[section]
\newtheorem{corollary}[theorem]{Corollary}
\newtheorem{lemma}[theorem]{Lemma}
\newtheorem{proposition}[theorem]{Proposition}

\theoremstyle{definition}
\newtheorem{definition}[theorem]{Definition}

\numberwithin{equation}{section}
\newtheorem*{theorem*}{Theorem} 

\newcommand{\R}{{\mathbb R}}
\newcommand{\Z}{{\mathbb Z}}
\newcommand{\rn}{{\mathbb R}^{n}}

\newcommand{\sint}{\strokedint}

\DeclareMathOperator{\III}{III}
\DeclareMathOperator{\II}{II}
\DeclareMathOperator{\I}{I}
\DeclareMathOperator{\dive}{div}

\providecommand{\abs}[1]{ \left \lvert  #1 \right \rvert }
\providecommand{\absgg}[1]{ \bigg \lvert  #1 \bigg \rvert }

\providecommand{\no}[1]{  \lVert  #1  \rVert }
\providecommand{\nos}[1]{ \left \lVert  #1 \right \rVert }

\providecommand{\la}[1]{  \langle #1 \rangle}

\title[Sparse bounds]{Sparse gradient bounds for divergence form elliptic equations}
\date{\today}

\author{Olli Saari} 
\author{Hua-Yang Wang}
\author{Yuanhong Wei}

\address{Olli Saari: Departament de Matem\`atiques,
	Universitat Polit\`ecnica de Catalunya,
	Avinguda Diagonal 647, 08028 Barcelona,
	Catalunya, Spain}

\address{Institute of Mathematics of UPC-BarcelonaTech, Pau Gargallo 14, 08028 Barcelona, Catalunya, Spain}

\address{Centre de Recerca Matem\`atica, Edifici C, Campus Bellaterra, 08193 Bellaterra, Catalunya, Spain}

\address{Hua-Yang Wang: Academy of Mathematics and Systems Science, Chinese Academy of Sciences, Beijing 100190, P.R. China}
\address{University of Chinese Academy of Sciences, Beijing 100049, P.R. China}
\email{wanghuayang@amss.ac.cn}

\address{Yuanhong Wei: School of Mathematics, Jilin University, Changchun 130012, P.R. China}
\email{weiyuanhong@jlu.edu.cn}

  \makeatletter
  \@namedef{subjclassname@1991}{{\rm 2020} Mathematics Subject Classification}
  \makeatother

\usepackage[normalem]{ulem}
\begin{document}

\begin{abstract}
We provide sparse estimates for gradients of solutions to divergence form elliptic partial differential equations in terms of the source data.
We give a general result of Meyers (or Gehring) type,
a result for linear equations with VMO coefficients
and a result for linear equations with Dini continuous coefficients.
In addition,
we provide an abstract theorem conditional on PDE estimates available.
The linear results have the full range of weighted estimates with Muckenhoupt weights as a consequence.
\end{abstract}

\maketitle

\section{Introduction}
Let $\theta \in (0,1)$. 
A family of cubes $\mathcal{Q}$ in $\rn$, $n \ge 1$,
is said to be $\theta$-sparse if for each $Q \in \mathcal{Q}$
there exists $E_Q \subset Q$ such that $|E_Q| \ge \theta |Q|$
and for all $x \in \rn$
\[
\sum_{Q \in \mathcal{Q}} 1_{E_Q}(x) \le 1 .
\]
An operator $T$ acting on some test function class is said to satisfy a $\theta$-sparse 
$(s,r)$-bound if for each pair of test functions $(f,g)$ there exists a $\theta$-sparse family $\mathcal{Q}$
such that
\begin{multline}
\label{intro-sparse}
\left|\int Tf(x) g(x) \, dx \right| \\
\le C \sum_{Q \in \mathcal{Q}} |Q| \left( \frac{1}{|3Q|} \int_{3Q} |f(x)|^{s} \, dx \right)^{1/s}\left( \frac{1}{|3Q|} \int_{3Q} |g(x)|^{r} \, dx \right)^{1/r} .
\end{multline}
Despite its repelling appearance,
a sparse bound is a useful estimate.
Its validity implies immediately $L^{p}$ bounds for $s<p<r'$. 
Another immediate consequence is the quantitative weighted estimate,
which we detail in  Proposition \ref{prop-weigths-quote}, quoted from \cite{MR3531367}.

Sparse bounds were originally introduced in the context of quantitative weighted estimates for Calder\'on--Zygmund singular integrals by Lerner \cite{MR3127380}. 
Their use led to astonishingly simple proofs of the $A_{2}$ theorem \cite{MR3625108,MR3484688,MR4058547},
whose original proofs in the cases of Beurling--Ahlfors transform \cite{MR1894362},
the Hilbert transform \cite{MR2354322} and 
also in the general case \cite{MR2912709}
are somewhat challenging.
Lerner's original work involved so-called sparse domination in Banach function spaces norms,
which was later improved to pointwise bounds \cite{MR3521084,MR4007575}.
The bounds in terms of duality pairings as in \eqref{intro-sparse}
were first used in \cite{MR3531367},
where their applicability outside of the theory of singular integral operators was demonstrated.
Although weaker than pointwise bounds,
they are extremely important as they can be used to model $L^{p}$ bounds on a restricted range
unlike the pointwise bounds.
This opened the door for using the sparse approach to many operators previously out of reach.
Since the publication \cite{MR2354322}, 
sparse bounds have been proved for the spherical maximal function \cite{MR4041115},
pseudodifferential operators \cite{MR4094458}, 
bilinear Hilbert transforms and generalizations \cite{MR3873113,MR4310287} and many other operators.
The literature has become very extensive,
and it is impossible to give an exhaustive list.

The purpose of the present paper is to study sparse bounds thematically close but historically far 
from their origins.
We consider solutions to Dirichlet problems to 
elliptic partial differential equations of the form 
\[
\dive a(x,\nabla u(x)) = \dive F(x)
\]
where $F \in L^{2}_{loc}(\rn)$ is a source term
and $a$ a coefficient function as in Definition \ref{def:nonlinearity}.
Estimating $|\nabla u|$ in terms of $|F|$ is the topic of the other branch of Calder\'on--Zygmund theory.
The literature on various $L^{p}$ and weighted $L^{p}$ bounds
for various coefficient functions $a$ is again too vast to be listed.
In addition to the linear theory,
there has been a number of works studying related questions
for nonlinear equations \cite{MR0722254} and
nonlinear potential theory \cite{MR2823872,MR3004772}, 
and there have been works addressing more and more irregular coefficients such as \cite{MR2069724}.
To the best of our knowledge,
however,
not too many sparse bounds are known in the context of differential equations.
Results in \cite{MR3531367} do not deal with the exact setup we are interested in 
and the results in \cite{MR4094458} only deal with pseudodifferential operators whose symbols are too smooth to really be applied here. 
Hence we try to give an overview on phenomena in Calder\'on--Zygmund theory of linear and nonlinear equations from the sparse point of view. 

Our first main result is a sparse analogue of Meyers' estimate,
which deals with source data almost in the natural energy space.
Here the setup is completely nonlinear.
We do not assume any spatial smoothness on the coefficients
and consequently this is a result that does not have an obvious analogue 
in the point of view based on kernels.
See Section \ref{sec:preliminaries} for precise definitions.
\begin{theorem}
\label{intro-nl}
Let $\theta \in (0,1)$, $0 < \lambda \le \Lambda < \infty$, $Q_0$ be a cube,
and a $(\lambda,\Lambda)$-elliptic coefficient $a$ in $Q_0$ be given.
Let $\Omega \subset Q_0$ be a Lipschitz domain.  
There exists $q = q(\lambda,\Lambda,n,\Omega) > 2$ such that the following holds.

Let $p \in (2,q]$, $F \in L^{p}(3Q_0;\rn)$ and $f \in L^{p_{*}}(3Q_0)$ 
where $p_{*} =np/(n+p)$.
Let $u \in W^{1,2}_0(\Omega)$ be a solution to
\begin{align*}
\dive a(x,\nabla u(x)) &= \dive F(x) + f(x), \quad x \in \Omega 
\end{align*} 
and let $g \in L^{\infty}(3Q_0)$ be non-negative.

Then there exists a $(1-\theta)$-sparse family $\mathcal{P}$ of cubes $P$
such that for all $P \in \mathcal{P}$ it holds $3P \subset 3Q_0$ 
and
\begin{multline*}
\int_{\Omega} |\nabla u (x)|g(x) \, dx
\le C \sum_{P \in \mathcal{P}} |P| \la{|F|}_{3P,2}\la{g}_{P,q'} \\
 + C \sum_{P \in \mathcal{P}} |P| 3\ell(P) \la{|f|}_{3P,2_{*}} \la{g}_{P,q'} 
\end{multline*}
where $C= C(\lambda,\Lambda,n,q,\Omega,\theta)$.
\end{theorem} 
Note that by H\"older's inequality and a maximal function argument, 
the sparse estimate above implies 
\[
\no{\nabla u}_{L^{p}(\Omega)} \le C \no{F}_{L^{p}(\Omega)} + C \no{f}_{L^{p_{*}}(\Omega)}.
\]
Estimates of this latter type were first proved by Bojarski \cite{Bojarski1957},
Meyers \cite{Meyers1963}, Elcrat--Meyers \cite{Meyers1975} and Gehring \cite{Gehring1973}.
The $L^{p}$ theory for general nonlinear equations
with less restricted range of $p$ is quite well developed,
and we refer to \cite{MR2069724} as an example of a more general result with BMO coefficients.
When it comes to sparse estimates,
however,
a peculiarity of the linearization argument we use 
impedes us from advancing beyond the Meyers type result in the nonlinear setting.
No matter how smooth the coefficient function is, 
we end up with a linearization that has no a apriori smoothness beyond measurability.

In the setting of linear equations,
a duality argument allows us to make the range of admissible $(s,r)$
for a sparse bound a symmetric rectangle centred at $(1/2,1/2)$. 
We postpone a more quantitative statement to the bulk of the paper,
but here we give an example of what we can say about relatively good equations.
We start with the case where all $L^{p}$-bounds are known to hold.
Recall that a coefficient matrix $A=[A_{ij}]_{n\times n}$ is said to be of vanishing mean oscillation (see \cite{MR0377518}) if 
for each $A_{ij}$ with $1\le i,j \le n$
it holds 
\[
\lim_{s \to \infty} \mathcal{V}_{s}(A_{ij}) < \infty, \quad 
\lim_{s \to 0} \mathcal{V}_{s}(A_{ij}) = 0
\]
where 
\[
\mathcal{V}_s(A_{ij}) = \sup_{x \in \rn, r \le s } \inf_{c \in \R} \frac{1}{r^{n}} \int_{Q_r(x)} | A_{ij}(y) - c| \, dy.
\]
Here we denote  
\[
Q_r(x) =  \prod_{i=1}^{n} [x_i - 2^{-1}r, x_i + 2^{-1}r ].
\]
\begin{theorem}
\label{intro-lin}
Let $\theta \in (0,1)$, $0 < \lambda \le \Lambda < \infty$, $Q_0$ be a cube,
and a linear $(\lambda,\Lambda)$-elliptic coefficient $A$ in $Q_0$ be given. 
Let $\Omega \subset Q_0$ be a $C^{2}$-domain.
Assume that $A$ is of vanishing mean oscillation.
Let $q \in (2,\infty)$ be given.
Let $p \in (2,q]$ and $F \in L^{p}(3Q_0;\rn)$.
Let $u \in W^{1,2}_0(\Omega)$ be a solution to
\begin{align*}
\dive A(x)\nabla u(x) &= \dive F(x) , \quad x \in \Omega
\end{align*} 
and let $g \in L^{\infty}(3Q_0)$ be non-negative.

Then there exists a $(1-\theta)$-sparse family $\mathcal{P}$ of cubes $P \subset Q_0$
such that for all $P \in \mathcal{P}$ it holds $3P \subset 3Q_0$ and
\[
\int_{\Omega} |\nabla u (x)|g(x) \, dx
\le C \sum_{P \in \mathcal{P}} |P| \la{|F|}_{3P,q'}\la{g}_{P,q'} 
\]
where $C= C(A,n,q,\Omega,\theta)$.
\end{theorem} 
This is a sparse variant of the results of Di Fazio \cite{MR1405255},
Iwaniec--Sbordone \cite{MR1631658} and Kinnunen--Zhou \cite{MR1720770,MR1799417} (linear case only).
We have omitted the function $f$ on the source data for simplicity.
It could be included,
but additional restrictions on the exponents would appear
and make reading more difficult.

The sparse bound as above has a corollary easy to state in terms of Muckenhoupt weights.
Recall that a locally integrable function $w \ge 0$ is said to be an $A_{p}$ weight with $1<p<\infty$ if 
\[
[w]_{A_p} := \sup_{Q} \left( \frac{1}{|Q|}\int_{Q} w(x) \, dx \right)  \left( \frac{1}{|Q|}\int_{Q} w(x)^{-1/(p-1)} \, dx \right)^{p-1} < \infty ,
\]
where the supremum is taken over all axis parallel cubes in $\rn$.
A version of Corollary \ref{intro-lin-cor-vmo} with continuous coefficient matrix
can be found as Theorem 2.5 in \cite{MR3531368} 
and with VMO regular coefficient matrix as Theorem 2.1 in \cite{MR4221600}.
These smoothness conditions are not optimal,
and further generalizations can be found for instance in \cite{MR3812104},
where coefficients with small BMO norm are treated.

\begin{corollary}
\label{intro-lin-cor-vmo}
Let $A$ be a $(\lambda,\Lambda)$-elliptic coefficient matrix in $Q_0$.
Let $\Omega \subset Q_0$ be a $C^{2}$-domain. 
Assume that $A$ is of vanishing mean oscillation.
If $p \in (1,\infty)$, $w \in A_{p}$, $F \in L^{p}(\Omega,wdx)$
and $u \in W_0^{1,2}(\Omega)$ is a solution to
\begin{align*}
\dive A(x)\nabla u(x) &= \dive F(x)  , \quad x \in \Omega,
\end{align*} 
then
\[
\no{\nabla u}_{L^{p}(\Omega,wdx)} \le C \no{F}_{L^{p}(\Omega,wdx)}  
\]
where $C= C(A,p,[w]_{A_p},n, \Omega)$.
\end{corollary}
We do not have good control on how the bound depends on the $A_{p}$ constant.
We refer to Proposition \ref{prop-weigths-quote} for some estimates in this setting.

Next we turn to the setting of Dini continuous coefficients,
where we  get exactly as good results as in the singular integral theory of Dini continuous Calder\'on--Zygmund kernels. 
Recall that $\omega : [0,1) \to [0,\infty)$ is said to be a Dini function if 
there are constants $c_1,c_2 >0$ such that whenever $0<s< t <2s < 1$,
it holds 
\[
c_1 \omega(s) \le \omega(t) \le c_2 \omega(2s)
\]
and for all $t \in (0,1)$
\[
\int _{0}^{t} \frac{\omega(s)}{s} \, ds < \infty.
\]
A function $g : 3Q_0 \to \R$ is said to be of $\omega$-Dini mean oscillation if 
\[
\omega_g(r) = \sup_{x \in \Omega}  \sint_{Q_r(x) \cap \Omega} |g(y) - \la{g}_{Q_{r}(x)\cap \Omega}|\,dy
\]
is a Dini function. 
 
\begin{theorem}
\label{intro-smooth}
Let $\theta \in (0,1)$, $0 < \lambda \le \Lambda < \infty$ and let $Q_0$ be a cube.
Let $\Omega \subset Q_0$ be a $C^{2}$-domain.
Let a linear $(\lambda,\Lambda)$-elliptic coefficient $A$ with Dini mean oscillation in $\Omega$ be given.   
Assume $\omega_A(r) \le c |\log r|^{-2}$ for some $c$ and all $r \in (0,1/2)$.
Let $F \in L^{p}(3Q_0;\rn)$ for some $p \in (1,\infty)$.
Let $u \in W^{1,2}_0(\Omega)$ be a solution to
\begin{align*}
\dive A(x) \nabla u(x) &= \dive F(x) , \quad x \in \Omega
\end{align*} 
and let $g \in L^{\infty}(3Q_0)$ be non-negative.

Then there exists a $(1-\theta)$-sparse family $\mathcal{P}$ of cubes $P$
such that $3P \subset 3Q_0$ and
\[
\int_{\Omega} |\nabla u (x)|g(x) \, dx
\le C \sum_{P \in \mathcal{P}} |P| \la{|F|}_{3P,1}\la{g}_{P,1} 
\]
where $C$ depends on $(\lambda,\Lambda,n,\Omega,\theta)$ and the Dini data of $A$.
\end{theorem}

\begin{corollary}
\label{intro-lin-cor-hoelder}
Let $A$ be a $(\lambda,\Lambda)$-elliptic coefficient matrix in $3Q_0$. 
Let $\Omega \subset Q_0$ be a smooth domain.
Assume $\omega_A(r) \le c |\log r|^{-2}$ for some $c$ and all $r \in (0,1/2)$.
If $p \in (1,\infty)$, $w \in A_{p}$, $F \in L^{p}(Q_0,wdx)$
and $u \in W_0^{1,2}(\Omega)$ is a solution to
\begin{align*}
\dive A(x)\nabla u(x) &= \dive F(x)  , \quad x \in \Omega
\end{align*} 
then
\[
\no{\nabla u}_{L^{p}(\Omega,wdx)} \le C [w]_{A_p}^{\max \left( \frac{1}{p-1},1 \right)} \no{F}_{L^{p}(3Q_0,wdx)}  
\]
where $C$ depends on $(n,p,\lambda,\Lambda,\Omega)$ and the Dini data of $A$.
\end{corollary}


We conclude the introduction with a few words about our proofs.
We were mostly inspired by the arguments of Lerner \cite{MR3484688} and Lerner--Ombrosi \cite{MR4058547}.
After putting the PDE and singular integral theory quantities in a correspondence and finding all the necessary estimates,
our proof is very similar to theirs.
However,
our formulations go through boundary value problems rather than Poisson problems as 
would be closer to the spirit of their writing.
Although it might be possible to reduce the Dini theorem to a situation where their result can be applied as a black box,
we do not believe there is much gain of insight or brevity in attempting to do that.
Hence we stick to the way of Dirichlet problems.
The idea of using reverse H\"older inequalities as a workhorse comes from two sources.
In PDE,
our original plan and motivation was to follow the work of Caffarelli and Peral \cite{MR1486629}.
In the sparse context,
the reverse H\"older inequalities were exploited very efficiently by Lacey in \cite{MR4041115}.
However,
we note that several arguments become invalid when the coefficient $a$ appearing in the equation becomes too rough.
This is why we present two arguments,
one for the general setting (which carries us until the VMO theorem)
and another one which takes advantage of full regularity theory for equations with Dini continuous coefficients, in particular the estimates of Dong, Escauriaza and Kim \cite{MR3620893,MR3747493}.
 
\bigskip

\noindent
\textit{Acknowledgement.}
O. Saari is supported by Generalitat de Catalunya (2021 SGR 00087), 
Ministerio de Ciencia e Innovaci\'on and the European Union -- Next Generation EU (RYC2021-032950-I),  (PID2021-123903NB-I00) and the Spanish State Research Agency 
through the Severo Ochoa and María de Maeztu Program for Centers and Units of Excellence in 
R\&D (CEX2020-001084-M).
Y. Wei is supported by the National Natural Science Foundation of China (Grant No. 11871242, 11971060),  and Scientific Research Project of Education Department of Jilin Province (Grant No. JJKH20220964KJ).
Part of the research was carried out during the first author's 
visit to Jilin University,
the group in which he wishes to thank for its kind hospitality.
We feel very grateful to anonymous reviewers for their time and valuable suggestions about our manuscript.
\section{Preliminaries}
\label{sec:preliminaries}

\subsection{Notation and generalities}
We work in $\rn$ and its open subsets.
For a measurable set $E \subset \rn$,
we denote its $n$-dimensional Lebesgue measure by $|E|$.
Given a function $h \in L_{loc}^{1}(\rn)$,
a measurable set $E$ of finite measure and a number $s \in (0,\infty)$,
we denote 
\[
\sint_{E} h(x) \, dx := \frac{1}{|E|} \int_E h(x) \, dx,
\quad \la{h}_{E,s} := \left(\sint_{E} |h(x)|^{s} \, dx  \right)^{1/s}. \]
We write $\la{h}_{E} := \la{h}_{E,1}$.

By a cube we mean a cartesian product of $n$ equally long intervals.
Given a cube $Q$,
we denote its side length by  $\ell(Q) = |Q|^{1/n}$.
The standard set of dyadic cubes is 
\[
\mathcal{D} = \{ 2^{k} ([0,1)^{n} + j ) :  k \in \Z, \ j \in \Z^{n} \}.
\]
Given a cube $Q_0 \subset \rn$,
let $\Theta$ be the translation and scaling with $\Theta (Q_0) = [0,1)^{n}$.
We denote 
\[
\mathcal{D}(Q_0) = \{ \Theta^{-1}(Q) : Q \in \mathcal{D}, \ \Theta^{-1}(Q) \subset Q_0 \}.
\]
We refer to $\mathcal{D}(Q_0)$ as dyadic subcubes of $Q_0$.

Given a reference cube $Q_0$,
we define the standard and fractional maximal functions as  
\[
M f(x) = \sup_{P \in \mathcal{D}(Q_0)} 1_{P}(x) \la{f}_{3P}
\]
and for $s \in (0,\infty)$
\[
M_{s} f(x) = \sup_{P \in \mathcal{D}(Q_0)} 1_{P}(x) (3\ell(P))^{s} \la{f}_{3P}.
\]
The standard theory includes the bounds
\begin{equation}
\label{maximal-function-theorem}
\no{M}_{L^{1}(\rn) \to L^{1,\infty}(\rn)} \le  C_n < \infty, \quad 
\no{M}_{L^{p}(\rn) \to L^{p}(\rn)} \le  C_{p,n} < \infty
\end{equation}  
for all $p \in (1,\infty]$,
see for instance Theorem 2.1.6 in \cite{MR2445437},
and the fractional maximal function bounds 
\begin{equation}
\label{frac-maximal-function-theorem} 
\no{M_s}_{L^{p}(\rn) \to L^{q}(\rn)} \le C_{s,n,p} < \infty
\end{equation} 
for 
\[
\frac{s}{n} = \frac{1}{p} - \frac{1}{q}, \quad q > \frac{n}{n-s},
\]
see for instance Chapter 6 in \cite{MR2463316}.
We denote the upper and lower Sobolev conjugates as 
\[
p_{*} = \frac{pn}{n+p}, \quad p^{*} = \frac{pn}{n-p}
\]
and we denote the H\"older conjugate as $p' = p/(p-1)$
when $p \in (1,\infty)$ and $1' = \infty$ and $\infty' = 1$.


\subsection{Solutions}
We axiomatize the properties of the coefficients of the equations 
that we study.
The definition below is standard.

\begin{definition}
\label{def:nonlinearity}
Consider an open and connected set $\Omega$.
Let $0 < \lambda \le \Lambda < \infty$ be real numbers.
We call a function $a : \Omega \times \rn \to \rn$
a $(\lambda,\Lambda)$-elliptic coefficient in $\Omega$ if the following conditions hold.
\begin{itemize}
  \item $x \mapsto a(x,\xi)$ is measurable for all $\xi$, and $\xi \mapsto a(x,\xi)$ is continuously differentiable for all $x$.
  \item for all $x \in \Omega$ and $\xi, \eta \in \rn$, we have $a(x,0) = 0$ and
\begin{equation}
\label{boundedness}
  |a(x,\xi) - a(x,\eta)| \le \Lambda |\xi-\eta|.
  \end{equation}
  This is referred to as \textit{Lipschitz condition in gradient variable}.
  \item for all $x \in \Omega$ and $\xi,\eta \in \rn$, we have
\begin{equation}
\label{strong-monotonicity}
(a(x,\xi) - a(x,\eta)) \cdot (\xi - \eta) \ge \lambda |\xi - \eta|^{2}.
\end{equation} 
This is referred to as \textit{strong monotonicity}.
\end{itemize}
\end{definition}

The next definition of a solution is standard. 

\begin{definition}
\label{def:solution}
Let $a$ be a $(\lambda,\Lambda)$-elliptic coefficient in $\Omega$.
Let $F \in L_{loc}^{1}(\Omega;\rn)$, $f \in L_{loc}^{1}(\Omega)$ and $h \in W^{1,2}(\Omega)$.
We call a function $u \in W^{1,2}(\Omega)$
a solution to 
\begin{align}
\label{equations}
\dive a(x,\nabla u(x)) &= \dive F(x) + f(x), \quad x \in \Omega \\
                    u(x) &= h(x), \quad x \in \partial \Omega \nonumber
\end{align}
if $u - h \in W_{0}^{1,2}(\Omega)$ and it holds for all test functions $\varphi \in C_c^{\infty}(\Omega)$  
\[
\int_{\Omega} a(x,\nabla u (x)) \cdot \nabla \varphi(x) \, dx = \int _{\Omega} F(x) \cdot \nabla \varphi(x) \, dx + \int _{\Omega} f(x) \varphi(x) \, dx .
\]
\end{definition}
We remark that by approximation,
the smoothness condition on test functions can always be relaxed down to one Sobolev derivative 
in the relevant $L^{p}(\Omega)$ space. 

The key property of solutions to equations with zero right hand side
is that their gradients satisfy a reverse H\"older inequality.
The strength of the reverse H\"older inequality for the gradients of solutions to homogeneous equations determines the quality of sparse, weighted and $L^{p}$ bounds that the general solutions satisfy.
The exponent on the right hand side is usually taken to be $1$ or $2$ instead of the $1/2$ here,
but it is a general property of reverse H\"older inequalities for solutions
that if they hold for one exponent on the right hand side,
then they hold for all them, see Appendix B of \cite{MR3565414}.
This is in strong contrast with the exponent on the left hand side,
which cannot be changed freely.

\begin{definition}
\label{def:rhi}
Let $\Omega$ be a domain and 
let $w \in L_{loc}^{1}(\Omega)$ be non-negative.
Let $q > 2$.
We say $w$ satisfies a $q$-weak reverse H\"older inequality in $\Omega$ 
if there exists a constant $C_{w,q}$ such that for all cubes $P \subset \rn$ 
\[
\left( \sint_{P} 1_{\Omega}(x)w(x)^{q} \, dx \right)^{1/q} \le C_{w,q} \left( \sint_{2P} 1_{\Omega}(x)w(x)^{1/2} \, dx \right)^{2} .
\] 
If this only holds for $P$ such that $2P \subset \Omega$,
we say $w$ satisfies a local $q$-weak reverse H\"older inequality in $\Omega$.
Similarly, if this only holds for $P$ such that $P \cap \Omega \ne \varnothing$,
we say $w$ satisfies a boundary $q$-weak reverse H\"older inequality in $\Omega$.
\end{definition}

For the proof of the sparse bound,
we will need a reverse H\"older inequality for gradients of differences of solutions
rather than solutions themselves.
For linear equations,
there is of course no difference,
but in the case of nonlinear equations,
this seems to be an important point.
To keep the exposition of linear and nonlinear cases unified,
we set a definition to capture this discrepancy.

To state the next definition,
we introduce a family of smoothened cubes.
We let $O_0$ be a smooth domain such that 
\[
\frac{299}{100} [-1/2,1/2]^{n} \subset O_0 \subset 3[-1/2,1/2]^{n}.
\]
Given any cube $Q \subset \rn$,
we let $O_Q$ be the smooth domain such that if $A_Q$ is the scaling and translation mapping $A_Q Q = [-1/2,1/2]^{n}$,
then $O_Q = A_Q^{-1}O_0$. 
Heuristically,
$O_Q$ is $3Q$ but with corners smoothed away.

\begin{definition}[Admissible upper exponent]
\label{def:upper_admissible}
Let $\Omega$ be a domain.
Let $a$ be a $(\lambda,\Lambda)$-elliptic coefficient in $\Omega$.
Given a cube $Q$,
define the class $\mathcal{U}(Q)$
as the family of pairs $(u,v)$ such that $u,v \in W^{1,2}_{0}(\Omega)$ 
and
\[
\dive a(x,\nabla u (x)) = \dive a(x,\nabla v (x))
\]
holds in $O_Q \cap \Omega$ in the sense of Definition \ref{def:solution}.

A number $q > 2$ is called an admissible upper exponent for $a$ if 
for any cube $Q$ and any pair $(u,v) \in \mathcal{U}(Q)$
the function 
\[
1_{2Q}|\nabla u - \nabla v |
\]
satisfies a local or boundary $q$-weak reverse H\"older inequality in $\Omega \cap 2Q$.
For an admissible upper exponent $q > 2$, we denote
\begin{align*}
N^{\Omega}_{h,bdr}(a,q) &= \sup_{Q} \sup_{(u,v) \in \mathcal{U}(Q)} \sup_{P \subset 2Q: P \cap \Omega \ne \varnothing} \frac{\la{|\nabla u - \nabla v|}_{P,q}}{\la{|\nabla u - \nabla v|}_{2P,1/2}}, \\
N^{\Omega}_{h,loc}(a,q) &= \sup_{Q} \sup_{(u,v) \in \mathcal{U}(Q)} \sup_{P \subset 2Q: 2P \subset \Omega} \frac{\la{|\nabla u - \nabla v|}_{P,q}}{\la{|\nabla u - \nabla v|}_{2P,1/2}} .
\end{align*}
where the supremum in $Q$ is over all cubes.
\end{definition}
Notice that $N^{\Omega}_{h,bdr}$ may be infinite a priori.
For concrete applications of our results,
verifying finiteness of this quantity is separated 
from the main sparse argument of Section \ref{s:general}.
When proving the theorems in the introduction,
we have three cases.
For Theorem \ref{intro-nl},
we rely on global Meyers type estimates valid in Lipschitz domains,
which we state and prove as Proposition \ref{prop-admissible-1}. 
For Theorem \ref{intro-lin},
we use the boundary regularity results of Kinnunen and Zhou \cite{MR1799417}.
Finally,
for Theorem \ref{intro-lin-cor-hoelder},
we can take advantage of the boundary regularity estimates of Escauriaza, Dong and Kim \cite{MR3747493}.

The second key property of solutions is an energy estimate.
An energy estimate as below is always valid for $q = 2$,
but in order to carry out a bootstrap argument,
we write a generic definition. 
Unlike in the case of admissible upper exponent,
the case of lower exponents is interesting only in the case of estimates holding up to the boundary.
Local versions do not suffice for running our arguments.

\begin{definition}[Admissible lower exponent]
\label{def:lower_admissible}
Let $a$ be a $(\lambda,\Lambda)$-elliptic coefficient in $\Omega$.
A number $1 < q \le 2$ is called an admissible lower exponent for $a$ if
there exists a finite number $N^{\Omega}_l(a,q)$ such that 
for all cubes $P$,  
for any source data
$F \in L_{loc}^{q}(\Omega)$ and $f \in L_{loc}^{q_{*}}(\Omega)$,
and for any solution $u \in W^{1,2}_{0}(O_P \cap \Omega)$ to \eqref{equations} in $O_P \cap \Omega$ 
it holds 
\begin{multline*}
\no{   \nabla u }_{L^{q}(O_{P} \cap \Omega)} \\
 \le N^{\Omega}_l(a,q) \left(
 \int_{O_{P}\cap \Omega} |F(x)|^{q} \, dx
+  \left( \int_{O_{P}\cap \Omega}  |f(x)|^{q_{*}} \, dx \right)^{q/q_{*}} \right)^{1/q}.  
\end{multline*}
If $q_{*} \le 1$,
we require $f \equiv 0$.
\end{definition}
 
\section{A general result}
\label{s:general}
In this section,
we prove the following general result.

\begin{theorem}[Local sparse bound]
\label{thm-nl-technical-version}
Let $\theta \in (0,1)$, $0 < \lambda \le \Lambda < \infty$ and $Q_0$ be a cube.
Let a $(\lambda,\Lambda)$-elliptic coefficient $a$ in $3Q_0$ be given.
Let $q_l \in (1,2]$ be an admissible lower exponent for $a$ 
in the sense of Definition \ref{def:lower_admissible}.
Let $q_h \in (2,\infty)$ be an admissible upper exponent for $a$ 
in the sense of Definition \ref{def:upper_admissible}.
Denote 
\[
A = N^{\Omega}_{l}(a,q_l), \quad B = N^{\Omega}_{h,loc}(a,q_h).
\]
Let $p > q_l$ and assume that $F \in L^{p}(3Q_0;\rn)$ and $f \in L^{p_{*}}(3Q_0)$.
If $f$ is not identically zero, 
assume $q_{l*} > 1$.
Let $u \in W^{1,2}_0(3Q_0)$ be a solution to
\begin{align*}
\dive a(x,\nabla u(x)) &= \dive  F(x) +  f(x), \quad x \in 3Q_0  
\end{align*}
according to Definition \ref{def:solution}.
Let $g \in L^{\infty}(3Q_0)$ be non-negative.

Then there exists a $(1-\theta)$-sparse family $\mathcal{P}$ of cubes $P$
such that for all $P \in \mathcal{P}$ it holds $3P \subset 3Q_0$ and
\begin{multline*}
\int_{Q_0} |\nabla u (x)|g(x) \, dx
\le C \sum_{P \in \mathcal{P}} |P| \la{|F|}_{3P,q_l}\la{g}_{P,q_h'} \\
 + C \sum_{P \in \mathcal{P}} |P| 3\ell(P) \la{|f|}_{3P,q_{l*}} \la{g}_{P,q_h'} 
\end{multline*}
where 
\[
C =   9^{n/q_l}  (A\Lambda+1)AB  \left(   \frac{\no{M}_{L^{1}(\rn) \to L^{1,\infty}(\rn)} }{\theta} \right)^{1/q_l}  .
\]
\end{theorem}

For the totality of this section,
we fix the notation as in the hypothesis of Theorem \ref{thm-nl-technical-version}. 
Clearly $a$ will be a $(\lambda,\Lambda)$-elliptic coefficient in all subcubes $Q \subset 3Q_0$.
For such a cube $Q$, 
we denote by $u_Q$ the solution $u \in W_{0}^{1,2}(O_Q)$ 
to 
\[
\dive a(x,\nabla u(x)) = \dive F(x) + f(x)
\]
in $O_Q$.
We denote 
\[
L(Q) := \int_{Q} |\nabla u_Q(x)| g(x) \, dx .  
\]
The main ingredient of the proof is an iteration formula for this quantity.

\begin{lemma}
\label{lemma-iteration-step}
Let $Q$ be a cube, $\theta \in (0,1)$, and let notations and assumptions of Theorem \ref{thm-nl-technical-version} remain valid. 
Then there exists a family $\mathcal{P}$ of pairwise disjoint cubes 
such that $3P \subset 3Q$ for all $P \in \mathcal{P}$;
it holds 
\begin{equation}
\label{iteration-claimed-inequality}
L(Q)
\le C |Q| (\la{|F|}_{3Q,q_l} + 3\ell(Q) \la{|f|}_{3Q,q_{l*}} )\la{g}_{Q,q_h'} + \sum_{P \in \mathcal{P}} L(P) 
\end{equation}
with 
\[
C = 9^{n/q_l}    ( A \Lambda + 1 ) A B \left(   \frac{\no{M}_{L^{1}(\rn) \to L^{1,\infty}(\rn)} }{\theta} \right)^{1/q_l}  
\]
and 
\[
\abs{Q \setminus \bigcup \mathcal{P}} \ge (1-\theta) |Q|.
\]
\end{lemma}

\begin{proof} 
For brevity,
we write $u = 1_{O_Q} u_Q$. 
Then $u \in W^{1,2}(\rn)$ as $u_Q \in W_{0}^{1,2}(O_Q)$.
Consider the set  
\[
\Xi = \{ x \in Q : M ( |\nabla u|^{q_l})(x) > D^{q_l} \}.
\] 
Then 
\[
 \Xi \cup \mathcal{N} \supset \{ x \in Q : |\nabla u(x)|^{q_l} > D^{q_l} \}  
\]
for some $\mathcal{N} \subset \rn$ with $|\mathcal{N}| = 0$.
Denote
\[
D  = A  \left(   \frac{  \no{M}_{L^{1}(\rn) \to L^{1,\infty}(\rn)}  }{ 3^{-n}  \theta} \right)^{1/q_l}  (\la{|F|}_{3Q,q_l} + 3\ell(Q) \la{|f|}_{3Q,q_{l*}} ) . 
\]  
We let $\mathcal{P}$ be the family of maximal cubes $P \in \mathcal{D}(Q)$ such that 
\[
\la{ |\nabla u|^{q_l}}_{3P} > D^{q_l}.
\]
Then $\bigcup \mathcal{P} = \Xi$ and for all $P \in \mathcal{P}$,
it holds $\la{ |\nabla u|^{q_l}}_{3P} \le 2^{n} D^{q_l}$. 
By \eqref{maximal-function-theorem} and the definition of the admissible lower exponent
(Definition \ref{def:lower_admissible})
\begin{multline*}
|\Xi| 
\le 
\frac{\no{M}_{L^{1}(\rn) \to L^{1,\infty}(\rn)}}{D^{q_l} } \int_{3Q} |\nabla u(x)|^{q_l} \, dx \\
    \le \frac{ A^{q_l} \no{M}_{L^{1}(\rn) \to L^{1,\infty}(\rn)}}{D^{q_l}} \left( \int_{3Q} |F(x)|^{q_l} \, dx +  \left( \int_{3Q}  |f(x)|^{q_{l*}} \, dx \right)^{q_l/q_{l*}}  \right) \\
\le \theta |Q|.
\end{multline*} 
In addition,
the family $\mathcal{P}$ is pairwise disjoint by maximality.
Hence it remains to prove the claimed inequality \eqref{iteration-claimed-inequality} for $L(Q)$.

We write
\begin{multline*}
\int_{Q} |\nabla u(x)| g(x) \, dx 
\le \int_{Q \setminus \bigcup \mathcal{P}} |\nabla u(x)| g(x)\, dx \\
+ \sum_{P \in \mathcal{P}}  \int_{P} |\nabla u(x) - \nabla u_P(x)| g(x)\, dx 
+ \sum_{P \in \mathcal{P}}   \int_{P} | \nabla u_P(x)| g(x)\, dx \\
= \I + \II + \III.
\end{multline*}  

Now 
\[
\I \le D  \int_{Q \setminus \bigcup \mathcal{P}}  g(x) \, dx  \le D |Q|  \la{g}_{Q,q'}
\]
by the fact $Q \setminus \bigcup \mathcal{P} = Q\setminus \Xi$ and by
H\"older's inequality.

To estimate $\II$,
we first notice that 
by H\"older's inequality 
\[
\int_{P} |\nabla u_P(x) - \nabla u(x)| g(x) \, dx \le \no{\nabla u_P - \nabla u}_{L^{q_h}(P)} \no{g}_{L^{q_h'}(P)}.
\]  
By construction,
the equation
\begin{equation}
\label{eq-abst-proof-defs}
\dive a(x, \nabla u_P(x)) = \dive \tilde{F}(x)
\end{equation}
with $\tilde{F}(x) = a(x, \nabla u(x))$ holds in $O_P$.
Hence $u$ and $u_{P}$ are a pair in $\mathcal{U}(P)$ as in Definition \ref{def:upper_admissible}. 
By the fact that $q_h$ is an admissible upper exponent in the sense of Definition \ref{def:upper_admissible},
we invoke that definition to estimate
\begin{multline*}
\no{\nabla u_P - \nabla u}_{L^{q_h}(P)}  \le B |P|^{1/q_h} \la{|\nabla u_P - \nabla u|}_{2P,q_l} \\
\le B |P|^{1/q_h}( \la{|\nabla u|}_{2P,q_l} +  \la{|\nabla u_P|}_{2P,q_l}) .
\end{multline*} 
Applying Definition \ref{def:lower_admissible} of admissible lower exponent to the equation \eqref{eq-abst-proof-defs} 
we bound further
\[
\la{|\nabla u_P|}_{2P,q_l} \le (3/2)^{n/q_l}   A \Lambda \la{|\nabla u|}_{3P,q_l}.
\]
Hence by maximality of the cubes $P$
\[
\no{\nabla u_P - \nabla u}_{L^{q_l}(2P)}  \le    3^{n/q_l}  (A \Lambda  + 1)B   D
\]
and by  H\"older's inequality
\[
\II \le  3^{n/q_l}    (A\Lambda + 1) B    D \sum_{P \in \mathcal{P}} |P|^{1/q_h} \no{g}_{L^{q_h'}(P)} \\
\le  3^{n/q_l}   (A\Lambda + 1) B D     |Q| \la{g}_{Q,q_h'}.
\]

Finally,
by definition  
\[
\III = \sum_j L(P_j).
\]
This concludes the proof.   
\end{proof}

\begin{proof}[Proof of Theorem \ref{thm-nl-technical-version}] 
We construct the sparse family by a recursion.
Let $\mathcal{S}_0 = \{Q_0\}$.
For $j \ge 0$ given,
we apply Lemma \ref{lemma-iteration-step} to each $Q \in \mathcal{S}_j$
to recover a pairwise disjoint family $\mathcal{P}(Q)$ as in Lemma \ref{lemma-iteration-step}.
Let $\mathcal{S}_{j+1} = \bigcup_{Q \in \mathcal{S}_{j}} \mathcal{P}(Q)$.
By induction,
we have constructed a $\mathcal{S}_{j}$ for each $j \ge 0$.
Denote $\mathcal{P} = \bigcup_{j=0}^{\infty} \mathcal{S}_j$.
By Lemma \ref{lemma-iteration-step},
for all $j \ge 0$ 
\begin{multline*}
\int_{Q_0} |\nabla u (x)|g(x) \, dx 
= L(Q_0) \\
\le C \sum_{k=0}^{j} \sum_{P \in \mathcal{S}_k} |P| (\la{|F|}_{3P,q_{l}} + 3\ell(Q) \la{|f|}_{3Q,q_{l*}} )\la{g}_{P,q_h'} 
+ \sum_{P \in \mathcal{S}_{j+1}} L(P).
\end{multline*}
To complete the proof,
we show that the family $\mathcal{P}$ is sparse and 
that the right-most sum tends to zero as $j \to \infty$.

For $P \in \mathcal{S}_{j}$, 
we set 
\[
E_P := P \setminus \bigg( \bigcup_{Q \in \mathcal{S}_{j+1}} Q \bigg) 
= P \setminus \bigg( \bigcup_{k=j+1}^{\infty} \bigcup_{Q \in \mathcal{S}_{k}} Q \bigg) 
\]
where the second equality follows by the construction:
a cube $Q \in \mathcal{S}_{k}$ with $k \ge j+1$ is a subcube of a cube $Q'$ in $\mathcal{S}_{j+1}$. 
It is then clear that 
\[
\sum_{P \in \mathcal{P}} 1_{E_P} \le 1.
\]
Further,
by Lemma \ref{lemma-iteration-step},
for $P \in \mathcal{S}_{j}$
\[
|E_P| = \absgg{P \setminus \bigg( \bigcup_{Q \in \mathcal{S}_{j+1}} Q \bigg) } \ge (1-\theta) |P|.
\]
Hence $\mathcal{S}$ is $(1-\theta)$-sparse.

By the definition of the sets $E_P$,
for each $j \ge 0$  
\begin{multline*}
\absgg{\bigcup_{P \in  \mathcal{S}_{j+1}} P }
= \absgg{\bigcup_{P \in  \mathcal{S}_{j}} \bigcup_{\substack{P' \subset P \\ P' \in \mathcal{S}_{j+1}}} P'   } 
= \absgg{\bigcup_{P \in  \mathcal{S}_{j}} P \setminus E_P  } 
\le \sum_{P \in \mathcal{S}_{j}} |P \setminus E_P| \\
=  \sum_{P \in \mathcal{S}_{j}} (|P | - | E_P|) 
\le \theta \sum_{P \in \mathcal{S}_{j}} |P|
= \theta \absgg{\bigcup_{P \in  \mathcal{S}_{j}} P }
\end{multline*}
where the last equality used that the cubes $P$ are pairwise disjoint 
by Lemma \ref{lemma-iteration-step}.
Iterating the estimate in the display above,
we conclude 
\[
\absgg{\bigcup_{P \in  \mathcal{S}_{j}} P } \le \theta^{j}|Q_0|.
\]
Finally
by H\"older's inequality and the definition of the lower admissible exponent (Definition \ref{def:lower_admissible})
\begin{multline*}
L(P) \le   \no{\nabla u_P}_{L^{q_l}(P)} \no{g}_{L^{q_l'}(P)} \\
  \le C \no{g}_{\infty} |P| 
\left( \la{|F|}_{3P,q_l}^{q_l} + 9\ell(P)^{q_l} \la{f}_{3P,q_{l*}}^{q_l}\right)^{1/q_l}  \\
  \le C \int_{P} (M( 1_{3Q_0} |F|^{q_l} )^{1/q_l} + M_1(1_{3Q_0}|f|^{q_{l*}})^{1/q_{l*}}).
\end{multline*} 
Summing and applying H\"older's inequality in the integral, we bound  
\begin{multline*}
\sum_{P \in \mathcal{S}_{j+1}} L(P) \\
\le  C \abs{\bigcup \mathcal{S}_{j+1}}^{1/p'} \nos{M( 1_{3Q_0} |F|^{q_l} )(x)^{1/q_l} + M_{q_{l*}}(1_{3Q_0}|f|^{q_{l*}}) (x)^{1/q_{l*}} }_{L^{p}(Q_0)} \\
\le K \theta^{j}
\end{multline*}
where the constant $K$ is finite by 
the maximal function theorem,
the fractional maximal function theorem 
and by the assumption on the source data.
Sending $j \to \infty$,
we conclude the proof.
\end{proof}

Next we discuss a global version of the sparse bound.
Under the current set of definitions,
the proof is essentially identical as the 
difference between local and boundary estimates is captured 
by the question whether or not the constant $N_{h,bdr}^{\Omega}(a,q)$ (as opposed to $N_{h,loc}^{\Omega}(a,q)$) from Definition \ref{def:upper_admissible} is finite.
To address the presence of the boundary of the domain,
we hence only have to modify the notation.
Given a cube $Q \subset 3Q_0$ and the reference domain $\Omega \subset Q_0$,
we let $U_Q$ be the smooth domain such that if $A_Q$ is a scaling and translation with $A_Q Q = [-1/2,1/2]^{n}$,
then $ U_Q = \Omega \cap A_Q^{-1}(O_0)$.
We denote by $u_Q^{\Omega}$ the solution $u \in W_{0}^{1,2}(U_Q)$ 
with source data $F$ and $f$.
We denote 
\[
L^{\Omega}(Q) := \int_{Q} |\nabla u_Q^{\Omega}(x)| g(x) \, dx .  
\]
Replacing $u_P$ by $u_{P}^{\Omega}$ in the proof of Theorem \ref{thm-nl-technical-version},
we see that the same argument works mutatis mutandis with the constant $B$ replaced by
$N^{\Omega}_{h,bdr}(a,q_h)$.
To apply Theorem \ref{thm-global-sparse-gen} to real problems,
one still has to prove that $N^{\Omega}_{h,bdr}(a,q_h)$ appearing on the right hand side of the sparse bound is finite,
and this is where more geometric considerations involving density of the complement of the domain 
or flattening the boundary appear.
That is done in Section \ref{s:applications} and Section \ref{s:dini-coefficients}.
More precisely,
the geometric input comes through the application of the main theorem in \cite{MR1799417} and Proposition \ref{prop-admissible-1}.
With this discussion,
we state the final abstract theorem.


\begin{theorem}[Global sparse bound]
\label{thm-global-sparse-gen}
Under the hypotheses of Theorem \ref{thm-nl-technical-version},
when $\Omega \subset Q_0$ is a domain and the additional assumption 
\[
N^{\Omega}_{h,bdr}(a,q_h) < \infty
\]
holds,
we have that
\begin{multline*}
\int_{\Omega} |\nabla u (x)|g(x) \, dx
\le C \sum_{P \in \mathcal{P}} |P| \la{|F|}_{3P,q_l}\la{g}_{P,q_h'} \\
 + C \sum_{P \in \mathcal{P}} |P| 3\ell(P) \la{|f|}_{3P,q_{l*}} \la{g}_{P,q_h'} 
\end{multline*}
where 
\begin{multline*}
C =   9^{n/q_l}  (N^{\Omega}_{l}(a,q_l)\Lambda+1)  N^{\Omega}_{l}(a,q_l)  N^{\Omega}_{h,bdr}(a,q_h) \\
\times  \left(   \frac{\no{M}_{L^{1}(\rn) \to L^{1,\infty}(\rn)} }{\theta} \right)^{1/q_l} .
\end{multline*}
\end{theorem}

\section{Applications}
\label{s:applications}

\subsection{Nonlinear equations}
In this section,
we apply the abstract Theorem \ref{thm-global-sparse-gen} to two concrete partial differential equations.
We start with the results for nonlinear equations. 
Our remaining task is to identify the admissible upper and lower exponents
according to Definition \ref{def:upper_admissible} and Definition \ref{def:lower_admissible}.

\begin{proposition}
\label{prop-energy-comparison}
Let $Q$ be a cube and $\Omega$ a domain.
Let $0 < \lambda \le \Lambda < \infty$, 
a $(\lambda,\Lambda)$-elliptic coefficient $a$ be given, 
and let $u \in W^{1,2}_0(O_Q \cap \Omega)$, $F \in L^{2}(O_Q ;\rn)$ and $f \in L^{2_{*}}(O_Q )$
be as in Definition \ref{def:solution}.
Then 
\[
\int_{O_Q \cap \Omega} |\nabla u (x)|^{2} \, dx  \le 
\frac{2}{\lambda^{2}} \int_{O_Q} |F(x)|^{2} \, dx
+ \frac{2c_n^{2}}{\lambda^{2}} \left( \int_{O_Q}  |f(x)|^{2_{*}} \, dx \right)^{2/2_{*}} .
\]
In other words, $2$ is an admissible lower exponent for $a$ in $\Omega$ with $N^{\Omega}_l(a,2) \le (1+c_n) \sqrt{2} /\lambda$ where $c_n$ is the constant from $W^{1,2}(\rn) \hookrightarrow L^{2^{*}}(\rn)$ Sobolev inequality.
\end{proposition}
\begin{proof}
We identify all the three functions $u$, $f$ and $F$ with
their extensions by zero to globally defined functions,
that is,
we set $u(x)=f(x)=F(x) = 0$ for all $x \notin O_Q \cap \Omega$.
Because $u \in W^{1,2}_{0}(O_Q \cap \Omega)$,
we have that the corresponding extension is in $W^{1,2}(\rn)$.
By strong monotonicity and the definition of solution  
\begin{multline*}
\lambda \int_{3Q} |\nabla u (x)|^{2} \, dx 
\le \int_{3Q} a(x,\nabla u(x)) \cdot \nabla u (x) \, dx \\
= \int_{3Q} F(x) \cdot \nabla u (x) \, dx +  \int_{3Q} f(x) u(x) \, dx .
\end{multline*}
By H\"older's inequality and 
by Sobolev inequality,
we have for the zeroth order term on the right hand side 
\begin{multline*}
\int_{3Q} f(x) u(x) \, dx \le \no{u}_{L^{2^{*}}(3Q)}\no{f}_{L^{2_*}(3Q)} 
  \le c_n \no{\nabla u}_{L^{2}(3Q)}\no{f}_{L^{2_*}(3Q)} \\
  \le \frac{\varepsilon}{2} \int_{3Q} |\nabla u (x)|^{2} \, dx 
+ \frac{c_n^{2} \no{f}_{L^{2_*}(3Q)}^{2}}{2 \varepsilon},   
\end{multline*}
where the last step followed for any $\varepsilon > 0$ from Young's inequality. 
The constant $c_n$ is the constant of Sobolev inequality.
By Young's inequality again,
we have for the first order term on the right hand side   
\[
\int_{3Q} F(x) \cdot \nabla u (x) \, dx 
\le \frac{\varepsilon}{2} \int_{3Q} |\nabla u(x)|^{2} \, dx + \frac{1}{2 \varepsilon} \int_{3Q} |F(x)|^{2} \, dx .
\] 
Choosing $\varepsilon = \lambda  /2$,
we conclude the proof.
\end{proof}

A classical result of Meyers \cite{Meyers1963}
shows that gradients of solutions to homogeneous problems satisfy a reverse H\"older inequality.
For completeness, we provide below a short argument for the inequality 
exactly in the form in which we need it. 
Note that this estimate holds up to the boundary.
We assume a Lipschitz condition on $\Omega$ to have control over $|3Q \cap \Omega^{c}|/|3Q|$
for cubes $Q$ with $2Q \cap \Omega^{c} \ne \varnothing$.
This gives further a control over the constant appearing in Sobolev inequality 
for functions $u$ on $3Q$ vanishing in $\Omega^{c}$.

\begin{proposition}
\label{prop-gehring}
Let $\Omega$ be a Lipschitz domain and let $Q_0$ be a cube.
Let $0 < \lambda \le \Lambda < \infty$, 
let $a$ be a $(\lambda,\Lambda)$-elliptic coefficient, 
and let $u \in W^{1,2}_0(\Omega)$.
Assume that $u \in W^{1,2}_0(\Omega)$ is a solution to 
\[
-\dive a(x,\nabla u(x)) = 0
\]
in $O_{Q_0} \cap \Omega$
with some boundary data (not necessarily zero at $\partial O_{Q_0} \setminus \partial \Omega$). 
Then there exists $q = q(n,\Lambda/\lambda,\Omega) > 2$ and $C = C(q,n,\Lambda/\lambda,\Omega)$ such that 
\[
\left( \sint_{Q_0}  1_{\Omega}(x)|\nabla u (x)|^{q} \, dx \right)^{1/q} \le  C\left( \sint_{2Q_0}  1_{\Omega}(x)|\nabla u (x)|^{1/2} \, dx \right)^{2}  .
\]
\end{proposition}
\begin{proof}
We identify $u$ with its extension by zero to a global Sobolev function.
Given any cube $Q$ such that $3Q \subset 3Q_0$,
let $\varphi \in C^{\infty}(\rn)$ be such that $\no{\nabla \varphi}_{\infty} \le 2\ell(Q)$
and $1_{Q} \le \varphi \le 1_{2Q}$.
Then $u \varphi^{2} \in W^{1,2}_{0}(\Omega)$ is a valid test function.
We obtain by strong monotonicity and product rule 
\begin{multline*}
\lambda \int |\nabla u(x)|^{2} \varphi(x)^{2} \, dx
\le \int a(x,\nabla u(x)) \cdot \nabla u(x) \varphi(x)^{2} \, dx \\
= \int a(x,\nabla u(x)) \cdot \nabla[ u(x) \varphi(x)^{2}] \, dx \\
- 2 \int u(x) \varphi(x) a(x,\nabla u(x)) \cdot \nabla  \varphi(x)  \, dx
\end{multline*} 
where the first term is zero by the equation and 
the second term is bounded by 
\[
\frac{1}{\varepsilon} \int u(x)^{2} |\nabla \varphi(x)|^{2} \, dx 
+ \varepsilon \Lambda^{2} \int \varphi(x)^{2} |\nabla u(x)|^{2} \, dx .
\]
Setting $\varepsilon = \lambda/(2\Lambda^{2})$,
we conclude 
\[
\int |\nabla u(x)|^{2} \varphi(x)^{2} \, dx \le \frac{4\Lambda^{2}}{\lambda^{2}} \int u(x)^{2} |\nabla \varphi(x)|^{2} \, dx =:T.
\]
If $2Q \subset \Omega$,
we may replace $u$ by $u-\la{u}_{2Q}$
as $u-\la{u}_{2Q}$ is a solution and $(u-\la{u}_{2Q})\varphi$ is a test function.
We may then apply Sobolev--Poincar\'e inequality on $T$.
If $2Q \cap \Omega^{c} \ne \varnothing$,
we may apply Sobolev inequality for Sobolev functions vanishing in an ample portion of $3Q$.
Altogether,
we obtain the bound 
\[
\left( \sint_{Q} 1_{\Omega}(x)|\nabla u(x)|^{2 } \, dx \right)^{1/2} 
\le \frac{2c_{n,\Omega} \Lambda}{\lambda}\left( \sint_{3Q} 1_{\Omega}(x)|\nabla u(x)|^{2_{*} } \, dx \right)^{1/2_{*}}  
\]
where $c_{n,\Omega}$ is larger of the constants of Sobolev--Poincar\'e inequalities with vanishing at the boundary or with mean value zero.
This estimate holds for all cubes $Q$ with $3Q \subset 3Q_0$.
Consequently,
we may apply Gehring's lemma (Theorem 3.22 in \cite{MR2867756}) together with the Appendix B of \cite{MR3565414} to conclude 
that there exists $q = q(\Lambda/\lambda,n,\Omega)$ and $C = C(q)$ such that for all cubes $Q$ with $3Q \subset 3Q_0$
\[
\left( \sint_{Q} 1_{\Omega}(x)|\nabla u(x)|^{q} \, dx \right)^{1/q} 
\le C \left( \sint_{2Q} 1_{\Omega}(x)|\nabla u(x)|^{1/2 } \, dx \right)^{2}.  
\]
This is the claimed inequality.
\end{proof}

We first remark that the number $q$ in the above proposition is an admissible upper exponent in the sense of Definition \ref{def:upper_admissible} for linear coefficient functions $a$ in $\Omega$ 
with $N^{\Omega}_{h,bdr}(a,q) \le C(n,q,\Lambda/\lambda,\Omega)$.
Indeed, the difference of two solutions to the same linear inhomogeneous equation
is a solution to the homogeneous linear equation with the same coefficient matrix,
and so we can apply the inequality to the difference appearing in the Definition \ref{def:upper_admissible}.
This proves that Definition \ref{def:upper_admissible} applies to the difference of a pair of solutions to any fixed linear equation.
For the case of nonlinear equations,
the difference of two solutions to the same inhomogeneous equation 
need not be a solution to the homogeneous equation with the same coefficient function.
However,
in case the nonlinear equation admits a linearization,
we can show that the difference is a solution to a homogeneous equation with similar structure. 
Precisely, we can prove the following.

\begin{proposition}
\label{prop-admissible-1}
Let $0 < \lambda \le \Lambda < \infty$.
Let $\Omega \subset \rn$ be a bounded Lipschitz domain and let $a$ be a $(\lambda,\Lambda)$-elliptic coefficient in $\Omega$.
Then,
in the sense of Definition \ref{def:upper_admissible}, 
there exists an admissible upper exponent $q = q(n,\Lambda/\lambda,\Omega) > 2$ and a finite constant $C= C(n,q,\Lambda/\lambda,\Omega)$ such that $N^{\Omega}_{h,bdr}(a,q) \le C$. 
\end{proposition}

\begin{proof}
Let $u$ and $v$ be as in the definition of admissible exponent.
Then 
\[
\dive a(x,\nabla u(x)) - \dive a(x,\nabla v(x)) = 0
\]
in the weak sense.
In case $a$ is linear,
we immediately see that $u-v$ is a solution and the claim follows from the Meyers' estimate,
Proposition \ref{prop-gehring}.
In case $a$ is not necessarily linear,
we use the fact that $a$ is continuously differentiable in the second variable.
Then
\begin{multline*}
a(x, \nabla u) - a(x, \nabla v) \\
= \int_{0}^{1} \nabla_{\xi} a(x, t\nabla u(x) + (1-t)\nabla v(x))(\nabla u(x) - \nabla v(x)) \, dt,
\end{multline*}
and we note that $u-v$ is solution to  
\begin{align*}
    \dive A \nabla (u-v) &= 0 
\end{align*}
where 
\[
A(x) = \int_{0}^{1} \nabla_{\xi} a(x, t\nabla u(x) - (1-t)\nabla v(x)) \, dt
\]
is uniformly elliptic matrix with entries measurable in $x$. 
Indeed, $x \mapsto t\nabla u(x) + (1-t)\nabla v(x)$ is measurable and 
and the Jacobian matrix of $a$ is continuous.
Hence the composition is measurable. 
We can invoke Meyers' estimate from Proposition \ref{prop-gehring}
to complete the proof.
\end{proof}

\begin{proof}[Proof of Theorem \ref{intro-nl}]
By Proposition \ref{prop-energy-comparison},
$2$ is an admissible lower exponent.
By Proposition \ref{prop-admissible-1},
there exists an upper admissible exponent $q > 2$
so that $N^{\Omega}_{h,bdr}(a,q)$ is finite.
By Theorem \ref{thm-global-sparse-gen},
the claim follows. 
\end{proof}
 
\subsection{Linear equations}
In the linear setting,
we dispose over several powerful tools that we have not been able to use 
in the case of nonlinear equations.
As we focus on the linear problem,
we fix a matrix $A : \Omega \to \R^{n\times n}$
such that for all $x,\xi \in \rn$
\[
a(x,\xi) = A(x) \xi.
\]
Then our assumptions on the coefficient $a$ convert into the standard ellipticity and boundedness assumptions on the matrix $A$ with measurable coefficients.
 
\begin{proof}[Proof of Theorem \ref{intro-lin}]
By Theorem 1.5 in Kinnunen--Zhou \cite{MR1799417},
we see that any $q > 2$ is an admissible upper exponent
and $N^{\Omega}_{h,bdr}(A,q)$ is finite,
that is,
the estimates hold up to the boundary.
By Theorem 2.1 in Di Fazio \cite{MR1405255},
we see that any $p>1$ is an admissible lower exponent.
The claim follows by Theorem \ref{thm-global-sparse-gen}. 
\end{proof}

\begin{proof}[Proof of Corollary \ref{intro-lin-cor-vmo}]
Given $w \in A_2$,
there exists $\varepsilon = \varepsilon([w]_{A_2},n) \in (0,1)$ such that  
it holds $w \in A_{2-\varepsilon}$ (Corollary 9.2.6 \cite{MR2445437}) and $w \in \rm{RH}_{1+\varepsilon}$ 
with $[w]_{\rm{RH}_{1+\varepsilon}} = C(n,[w_{A_2}])$ (Theorem 9.2.2 \cite{MR2445437}).
Choose $q' > 1$ so that $2/q' > 2-\varepsilon$ and  
$(q/q')' = 1/(2-q')< 1+ \varepsilon$.
Concretely, we can set $q' = 2/(2-\varepsilon^{2})$.
By Theorem \ref{intro-lin}, a $(q',q')$ sparse estimate holds.
Using the fact 
\[
\no{f}_{L^{2}(\R^{n},wdx)} = \sup_{\no{g}_{L^{2}(w^{-1}dx)} \le 1} \abs{ \int f(x)g(x) \, dx },
\]
we can invoke Proposition \ref{prop-weigths-quote},
to conclude from the sparse estimate above that 
\[
\no{1_{\Omega}\nabla u}_{L^{2}(\R^{n},wdx)} 
\le   C(n,A,\Omega,[w]_{A_2}) \no{F}_{L^{2}(\R^{n},wdx)}.
\] 
As this bound holds for all weights $w \in A_2$,
we can apply the Rubio de Francia extrapolation theorem (Theorem 9.5.3 in \cite{MR2463316}).
We hence conclude that for all $p$ and all $w \in A_p$
\[
\no{\nabla u}_{L^{p}(\Omega,wdx)}  
\le C(n,p,A,\Omega,[w]_{A_p}) \no{F}_{L^{p}(\rn,wdx)} .
\]
Restricting the attention to $F$ vanishing in $\Omega^{c}$
and using density of $C_c^{\infty}(\Omega)$ in $L^{p}(\Omega,wdx)$,
we conclude the proof.  
\end{proof}
 
\section{Case of Dini coefficients}
\label{s:dini-coefficients}
In this section,
we discuss the case where $A : \Omega \to \R^{n \times n}$
is assumed to be of Dini mean oscillation as defined in the introduction.
We only see this assumption through the results quoted from \cite{MR3747493}.
The following Proposition follows by Lemma 2.11 of \cite{MR3747493}
and the well-known properties of reverse H\"older inequalities,
see Appendix B in \cite{MR3565414}.

\begin{proposition}
Let $\Omega$ be a $C^{2}$-domain and let $v \in W_{0}^{1,2}(\Omega)$ be a solution to 
\[
\dive A(x) \nabla v(x) = 0
\]
in $O_P \cap \Omega$ and set $v = 0$ in $\Omega^{c}$.
Then 
\begin{equation}
\label{holder-to-boundary}
\no{\nabla v}_{L^{\infty}(P \cap \Omega)} 
\le C \left(\sint_{2P} 1_{\Omega}(x)|\nabla v(x)|^{1/2}\,dx \right)^{2}
\end{equation}
for $C = C(\lambda,\Lambda,n,{\rm Dini}_A,\Omega)$.
\end{proposition}

The second strong estimate we need is Theorem 1.9 from \cite{MR3747493}. 
\begin{proposition}
\label{prop:weak-type-DEK}
Let $\Omega$ be a $C^{2}$-domain and let $P$ be a cube. 
Let $F \in L^{2}(O_P \cap \Omega)$. 
Let $v \in W_{0}^{1,2}(O_P \cap \Omega)$ be a solution to 
\[
\dive A(x) \nabla v(x) = \dive F(x)
\]
in $O_P \cap \Omega$.
Assume $\omega_A(r) \le c |\log r|^{-2}$ for some $c$ and all $r \in (0,1/2)$.
Then for all $\mu > 0$ 
\begin{equation}
\label{weak-type}
|\{x \in O_P \cap \Omega : |\nabla v(x)| > \mu \}| \le \frac{C}{\mu} \int_{O_P \cap \Omega} |F(x)|\,dx
\end{equation}
for $C = C(\lambda,\Lambda,n,{\rm Dini}_A,\Omega)$.
\end{proposition}

Using the proposition above,
we can prove a weak type (1,1) bound for an auxiliary operator.
This operator takes the place of the auxiliary maximal function of Lerner and Ombrosi \cite{MR4058547}.
Given a cube $Q$,
we denote by $u_Q$ the function such that as a restriction $u_Q \in W_{0}^{1,2}(O_Q \cap \Omega)$ 
is a solution to 
\[
\dive A(x)\nabla u(x) = \dive F(x)
\]
in $O_Q \cap \Omega$ and $u_{Q} = 0$ in $(O_Q \cap \Omega)^{c}$.

\begin{proposition}
\label{aux-max-prop}
Fix a cube $Q$.
For a cube $P \subset O_Q \cap \Omega$,
we write $w_P = u_Q-u_P$.
We define an auxiliary function for all $x \in 3Q$ as
\[
S(x) = \sup_{P \in \mathcal{D}(3Q)} 1_{P}(x) \sup_{y,y' \in P} |\nabla w_P(y)- \nabla w_P(y')|.
\] 
Then there exists a constant $C = C(\lambda,\Lambda,n,{\rm Dini}_A,\Omega)$ such that 
for all $\mu > 0$
\begin{equation}
\label{s-function}
|\{x \in O_Q \cap \Omega : S(x) > \mu  \}| \le \frac{C}{\mu} \int_{O_Q \cap \Omega} |F(x)| \, dx .
\end{equation}
\end{proposition}

\begin{proof}
As always,
we extend all the functions appearing in the proof by zero to be globally defined functions.
Note that $w_P$ solves a homogeneous equation in $O_P \cap \Omega$.
By \eqref{holder-to-boundary},
for any $x \in O_Q \cap \Omega$ 
\begin{multline*}
S(x) \le C \sup_{P \in \mathcal{D}(3Q)} 1_{P }(x) \left( \sint_{2P} |\nabla w_P(y)|^{1/2} \, dy \right)^{2} \\
  \le C\sup_{P \in \mathcal{D}(3Q)} 1_{P }(x)\left( \sint_{2P} |\nabla u_P(y)|^{1/2} \, dy \right)^{2} + C M( 1_{O_Q \cap \Omega}|\nabla u|^{1/2}  )^{2}(x). 
\end{multline*}
By Kolmogorov's inequality (Exercise 2.1.5 in \cite{MR2445437}) and Proposition \ref{prop:weak-type-DEK} applied in $O_P$, the first term is bounded by maximal function of $F$.
Hence 
\begin{multline*}
\no{S}_{L^{1,\infty}(O_Q \cap \Omega)} \le C \no{M(1_{O_Q \cap \Omega}F)}_{L^{1,\infty}(O_Q \cap \Omega)} \\
+ C \no{M ( 1_{O_Q \cap \Omega} |\nabla u |^{1/2})^{2}}_{L^{1,\infty}(O_Q \cap \Omega)} .
\end{multline*}
Because the maximal function is bounded $L^{2,\infty}(\rn) \to L^{2,\infty}(\rn)$
and $L^{1}(\rn) \to L^{1,\infty}(\rn)$,
we conclude by an application of Proposition \ref{prop:weak-type-DEK} in $O_Q$
\begin{multline*}
\no{S}_{L^{1,\infty}(O_Q \cap \Omega)}  \le C\no{M(1_{O_Q \cap \Omega}F)}_{L^{1,\infty}(O_Q \cap \Omega)} + C  \no{  \nabla u }_{L^{1,\infty}(O_Q \cap \Omega)} \\
  \le C \int_{O_Q \cap \Omega} |F(x)| \, dx .
\end{multline*}
This completes the proof.
\end{proof}

Now we are in a position to prove the strong iteration lemma for $u$
solving an equation with Dini continuous coefficients.
We denote again
\[
L(Q) := \int_{Q} |\nabla u_Q(x)| g(x) \, dx .  
\]

\begin{lemma}
\label{lemma-iteration-step-smooth}
Let $\theta \in (0,1)$ and 
let $Q$ be a cube. 
Let $g\in L^{\infty}(3Q)$.
Then there exists a family $\mathcal{P}$ of pairwise disjoint cubes 
$P \subset 3Q$ such that  
\begin{equation}
\label{iteration-claimed-inequality-smooth}
L(Q)
\le C |Q| \la{|F|}_{3Q,1} \la{g}_{Q,1} + \sum_{P \in \mathcal{P}} L(P) 
\end{equation}
with $C =  C(\lambda,\Lambda,n,{\rm Dini}_A,\Omega,\theta)$
and 
\[
\abs{Q \setminus \bigcup \mathcal{P}} \ge (1-\theta) |Q|.
\]
\end{lemma}

\begin{proof} 
For brevity,
we write $u = u_Q$.
Consider the set  
\[
\Xi = \{ x \in Q : \max(|\nabla u(x)|,S(x),M(1_{3Q}F)(x)) > D \}  
\] 
where $D = 12^{n+1} \max(C_w,C_S,\|M\|_{L^{1}(\mathbb{R}^n)\rightarrow L^{1,\infty}(\mathbb{R}^n)}) \la{|F|}_{3Q,1}/\theta$
and $C_w$ is the constant from \eqref{weak-type}
and $C_S$ is the constant from \eqref{s-function}.
We let $\mathcal{P}$ be the family of maximal cubes $P \in \mathcal{D}(Q)$ such that 
\[
\frac{|P \cap \Xi|}{|P|} > 2^{-n-1}.
\]
Then 
\begin{equation}
\label{hp:cover}
\bigcup_{P \in \mathcal{P}} P \supset \Xi.
\end{equation}
Now
\[
|\bigcup \mathcal{P}|
\le \sum_{P \in \mathcal{P}} |P|
\le 2^{n+1}  \sum_{P \in \mathcal{P}} |P \cap \Xi|
= 2^{n+1} | \bigcup_{P \in \mathcal{P}} P \cap \Xi |
\le 2^{n+1} |\Xi|
\]
and by \eqref{weak-type}, \eqref{s-function} and the weak $(1,1)$ estimate of the maximal function,
\[
|\Xi| \le \frac{3(C_w + C_S +\|M\|_{L^{1}(\mathbb{R}^n)\rightarrow L^{1,\infty}(\mathbb{R}^n)})   }{D} \int_{3Q} |F(x)| \, dx \le \theta |Q|.
\]
Hence it remains to prove the claimed inequality \eqref{iteration-claimed-inequality-smooth} for $L(Q)$.

We write
\begin{multline*}
\int_{Q} |\nabla u(x)| g(x) \, dx 
\le \int_{Q \setminus \bigcup \mathcal{P}} |\nabla u(x)| g(x)\, dx \\
+ \sum_{P \in \mathcal{P}}  \int_{P} |\nabla u(x) - \nabla u_P(x)| g(x)\, dx 
+ \sum_{P \in \mathcal{P}}   \int_{P} | \nabla u_P(x)| g(x)\, dx \\
= \I + \II + \III.
\end{multline*}
Now by definition of the cubes $P$
\[
\I \le D  \int_{Q \setminus \bigcup \mathcal{P}}  g(x) \, dx  \le D |Q|  \la{g}_{Q,1}.
\] 

Next we estimate $\II$.
By maximality of $P$,
it always holds 
\begin{equation}
\label{max-density-cz-condition}
\frac{|P \cap \Xi|}{|P|} \le \frac{1}{2}.
\end{equation}
Let $N = 4 \cdot 3^{n} C_w$.
Recall that $u_P$ is identically zero in $(O_P \cap \Omega)^{c}$.
Hence by \eqref{weak-type}
\begin{multline*}
|\{ x \in P : |\nabla u_P(x)| > N D  \}| \\
\le |\{ x \in O_P \cap \Omega: |\nabla u_P(x)| > N D  \}| 
\le \frac{C_w}{D N} \int_{3P}|F(x)|\,dx  \\
\le \frac{3^{n}C_w}{D N} |P| \inf_{x \in P} M(1_{3Q}F)(x)
\le \frac{3^{n} C_w}{N} |P|
\le \frac{1}{4}|P|
\end{multline*}
where we used the fact $|P \setminus \Xi| \ge |P|/2$, following by \eqref{max-density-cz-condition}, 
to bound the maximal function.
Hence there exists $x_P \in P \setminus \Xi$
with 
\[
\max(|\nabla u_{P}(x_P)|,|\nabla u(x_P)|) \le N D.
\]

Denote $w_P = u_P - u$.
We write  
\begin{multline*}
\int_{P} |\nabla u_P(x) - \nabla u(x)| g(x) \, dx =
\int_{P} |\nabla w_P(x)|g(x) \,dx  \\
\le  |\nabla w_P(x_P)| \int_{P} g(x) \,dx +  \int_{P} |\nabla w_P(x) - \nabla w_P(x_P)|g(x) \,dx  .
\end{multline*}  
Here 
\begin{multline*}
\int_{P} |\nabla w_P(x) - \nabla w_P(x_P)|g(x) \,dx \le \inf_{x \in P} S(x) \la{g}_{P,1}|P|  \\
\le |P| S(x_P) \la{g}_{P,1} \le D |P| \la{g}_{P,1}
\end{multline*}
and 
\[
 |\nabla w_P(x_P)|
\le  |\nabla u_P(x_P)| +  |\nabla u(x_P)|
\le 2N D.
\]
This concludes the estimate for $\II$.

Finally,
by definition  
\[
\III = \sum_j L(P_j).
\]
This concludes the proof.   
\end{proof}

\begin{proof}[Proof of Theorem \ref{intro-smooth}]
The result follows by iterating Lemma \ref{lemma-iteration-step-smooth}.
See the iteration at the end of the proof of Theorem \ref{thm-nl-technical-version},
which is an identical argument.
\end{proof}

\begin{proof}[Proof of Corollary \ref{intro-lin-cor-hoelder}]
This follows at once by Theorem \ref{intro-smooth}
and Proposition \ref{prop-weigths-quote} with $r = 1 = s'$.
Indeed for $\sigma = w^{1-p'}$ and any $g \in L^{\infty}(\rn) \cap L^{p'}(\rn,\sigma dx)$,
we get 
\begin{multline*}
\int 1_{\Omega}|\nabla u(x)||g(x)|\, dx 
\le C \sum_{P} |P| \la{1_{3Q_0}|F|}_{3P,1} \la{1_{3Q_0}g}_{3P,1} \\
\le C[w]_{A_{p}}^{\max\left(1, \frac{1}{p-1}\right)} \no{1_{3Q_0}F}_{ L^{p}(\rn,w dx)}\no{g}_{ L^{p'}(\rn, \sigma dx)}.
\end{multline*} 
The claim follows by taking supremum over the family of all $g$ as above and satisfying $\no{g}_{L^{p'}(\rn,\sigma dx) } \le 1$.
\end{proof}

\section{Sparse forms}
We recall here how the sparse bounds imply weighted bounds.
Recall that a locally integrable $w \ge 0$ is in $RH_{s}$ with $s \ge 1$ 
if there is a constant $[w]_{RH_s}$ such that for all cubes $P \subset \rn$ 
\[
\left(\sint_{P} w(x)^{s} dx \right)^{1/s} \le [w]_{RH_s} \sint_{P} w(x) \, dx.
\] 
The following proposition is essentially Proposition 6.4 in \cite{MR3531367}.
Although the exact wording differs slightly from that in \cite{MR3531367},
an inspection of the proof quickly shows that this statement follows from the very same proof. 
For the reader's convenience,
we outline the short proof of the case $r=s'=1$ below.

\begin{proposition}
\label{prop-weigths-quote}
Assume $1 \le r \le 2 < s \le \infty$ and $\theta \in (0,1)$.
Let $p \in (r,s)$,
with the understanding $\infty ' = 1$ and $1' = \infty$.
Then there exists a finite constant $C$ such that the following holds.
If $\mathcal{P}$ is a $\theta$-sparse family,
then for all non-negative test functions $f,g \in C_c(\rn)$
\begin{multline*}
\sum_{P} |P| \la{f}_{3P,r} \la{g}_{3P,s'} \\
\le C ( [w]_{A_{p/r}}[w]_{RH_{(s/r)'}} ) ^{\beta} \no{f}_{L^{p}(\rn, wdx)} \no{g}_{L^{p'}(\rn, w^{1-p'}dx)}
\end{multline*}
where
\[
\beta = \max \left( \frac{1}{p-r}, \frac{s-1}{s-p} \right).
\] 
\end{proposition}

We do not need the exact quantitative form except for the proof of Corollary \ref{intro-lin-cor-hoelder}.
In that case,
we have $r = s'=1$. 
We reproduce the simple proof modulo quoting some basic results
in this special case. 
First,
by the Three Lattice Theorem (Theorem 3.1 in \cite{MR4007575}),
it suffices to bound 
\[
\sum_{P} |P| \la{f}_{P,1} \la{g}_{P,1} = \int \left(\sum_{P} 1_{P}(x) \la{f}_{P,1}\right) g(x) \, dx 
\]
with all cubes $P$ in the sparse family being dyadic.
By the Rubio de Francia extrapolation theorem (Theorem 9.5.3 in \cite{MR2463316}),
it suffices to deal with the case $p = 2$. 
Denoting $\sigma = w^{-1}$,
we write  
\[
\la{f}_{P,1} = \frac{\sigma(P)}{|P|} \left( \frac{1}{\sigma(P)} \int_{P} f(x)\sigma(x)^{-1}  \sigma(x)\, dx  \right) 
\] 
and 
\[
\la{g}_{P,1} = \frac{w(P)}{|P|} \left( \frac{1}{w(P)} \int_{P} g(x)w(x)^{-1}  w(x)\, dx  \right) .
\] 
Here for a measurable set $E$ and a non-negative locally integrable function $v$,
we denote 
\[
v(E) = \int_{E} v(x) \, dx.
\]
Define  
\[
M^{\sigma}f(x) := \sup_{P \in \mathcal{D}} 1_P(x)\la{f}_{P}^{\sigma} 
:=  \sup_{P \in \mathcal{D}} \frac{1_{P}(x)}{\sigma(P)} \int_{P} f(x) \sigma(x) \, dx .
\]
Then the sparse form can be bounded as
\begin{align*}
\sum_{P} &|P| \la{f}_{P,1} \la{g}_{P,1} \\
 &\le \sum_{P} |E_P| \frac{w(P)\sigma(P)}{|P|^{2}} \inf_{x \in E_P} [M^{\sigma}(f\sigma^{-1})(x) \cdot M^{w}(gw^{-1})(x) ]\\
&\le [w]_{A_2} \sum_{P} \int_{E_P} M^{\sigma}(f\sigma^{-1})(x) M^{w}(gw^{-1})(x) \, dx \\
&= [w]_{A_2}  \int  M^{\sigma}(f\sigma^{-1})(x)w(x)^{-1/2} \cdot M^{w}(gw^{-1})(x)w(x)^{1/2} \, dx \\
&\le [w]_{A_2} \no{M^{\sigma}(f\sigma^{-1})}_{L^{2}(\rn,\sigma dx)}
\no{M^{w}(gw^{-1})}_{L^{2}(\rn,w dx)} \\
&\le 4 [w]_{A_2} \no{f\sigma^{-1}}_{L^{2}(\rn,\sigma dx)}\no{gw^{-1}}_{L^{2}(\rn,w dx)} \\
&= 4 [w]_{A_2} \no{f }_{L^{2}(\rn,w dx)}\no{g }_{L^{2}(\rn,\sigma dx)} .
\end{align*}
The penultimate inequality used the inhomogeneous dyadic maximal inequality,
which holds for all weights with constant $2$ (see for instance Theorem 15.1 from \cite{MR4007575}).

\bibliography{references}
\bibliographystyle{abbrv}

\end{document}